\documentclass[10pt]{article}
\font\smallit=cmti10
\font\smalltt=cmtt10

\usepackage{caption}
\usepackage{color}
\usepackage{amssymb,latexsym,amsmath,epsfig,amsthm} 
\usepackage{empheq}
\usepackage{here}
\usepackage{ascmac}
\makeatletter
\usepackage{here}
\usepackage{comment}

\renewcommand\section{\@startsection {section}{1}{\z@}
{-30pt \@plus -1ex \@minus -.2ex}
{2.3ex \@plus.2ex}
{\normalfont\normalsize\bfseries\boldmath}}

\renewcommand\subsection{\@startsection{subsection}{2}{\z@}
{-3.25ex\@plus -1ex \@minus -.2ex}
{1.5ex \@plus .2ex}
{\normalfont\normalsize\bfseries\boldmath}}

\renewcommand{\@seccntformat}[1]{\csname the#1\endcsname. }

\makeatother
\newtheorem{theorem}{Theorem}
\newtheorem{lemma}{Lemma}

\theoremstyle{definition}
\newtheorem{definition}{Definition}[section]

\begin{document}

\begin{center}
\uppercase{\bf A Variant of Game of Sliding Coins}
\vskip 20pt
{\bf Ryohei Miyadera }\\
{\smallit Keimei Gakuin Junior and High School, Kobe City, Japan}\\
{\tt runnerskg@gmail.com}
\vskip 10pt
{\bf Hikaru Manabe}\\
{\smallit Keimei Gakuin Junior and High School, Kobe City, Japan}\\
{\tt urakihebanam@gmail.com}
\vskip 10pt
{\bf Unchon Lee}\\
{\smallit Keimei Gakuin Junior and High School, Kobe City, Japan}\\
{\tt infinitemenmae@gmail.com}
\vskip 10pt
\end{center}
\vskip 20pt
\centerline{\smallit Received: , Revised: , Accepted: , Published: } 
\vskip 30pt

\centerline{\bf Abstract}
\noindent
Here, we present a variant of the sliding coins game. Two coins are placed on distinct squares of a semi-infinite linear board with squares numbered $0, 1, 2, dots, $. Two players take turns and move a coin to a lower unoccupied square. When a coin is pushed to the outside of the linear board, the players cannot use this coin anymore. 
In this game, we have another operation of moving coins: moving the coin on the right leftward and pushing the coin on the left. This last operation complicates this game's mathematical structure, but we have managed to make formulas for Grundy numbers. Since all the positions of this game of two coins are the next player's winning position, this game of two coins is trivial as a game. However, by making the sum of two games, we get a meaningful game in which the player who plays for the last time is the winner. 
With these Grundy numbers formulae, we get the winning strategy.

\pagestyle{myheadings} 
\markright{\smalltt   (  )\hfill} 
\thispagestyle{empty} 
\baselineskip=12.875pt 
\vskip 30pt


\section{A game of sliding coins}\label{introductionsection}
Let $\mathbb{Z}_{\geq0}$ and $\mathbb{N}$ be the sets of nonnegative and natural numbers, respectively.

\begin{definition}\label{definitionofstrap}
There are two coins, and two players take turns and move one of the coins to leftward as in Figures \ref{twocoin1} and \ref{twocoin2}, or move the coin on the right to leftward and push also the coin on the left leftward as in Figure \ref{twocoin3}.
Players can also push one of the coins or both coins over the edge as in Figures \ref{twocoin4} and \ref{twocoin5a}. 
Players cannot move the coin on the right over the coin on the left. See Figure \ref{twocoin5}
\end{definition}

\begin{definition}\label{definitionofstrap2}
We use two  of the game in Definition \ref{definitionofstrap} to make the game of four coins in Figure \ref{twodirectiongame}. Players take turns and move coins on the left strap leftward, and coins on the right strap rightward. The player who plays for the last time is the winner.
\end{definition}

\begin{figure}[H]
\begin{tabular}{cc}
\begin{minipage}[t]{0.43\textwidth}
\begin{center}
\includegraphics[height=0.9cm]{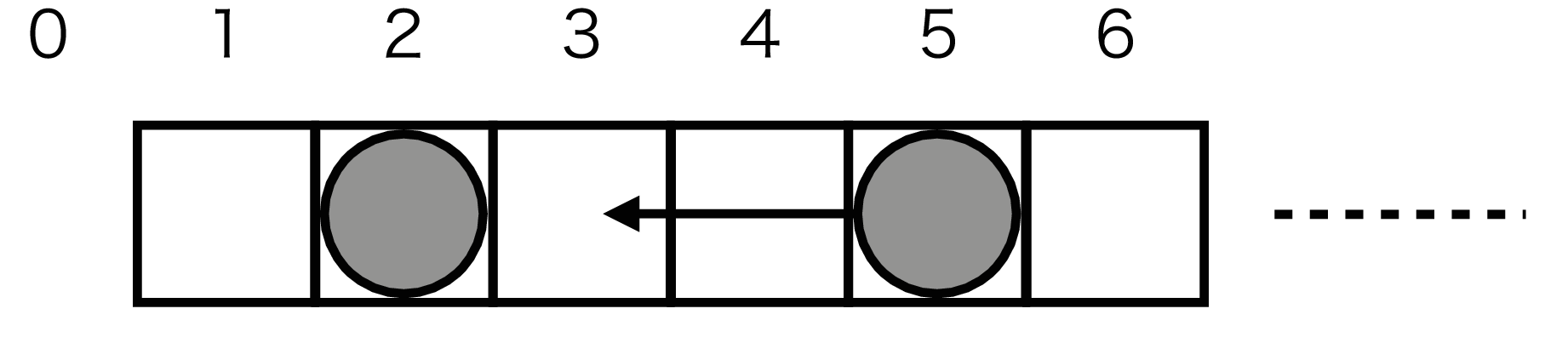}
\caption{}
\label{twocoin1}
\end{center}
\end{minipage}
\begin{minipage}[t]{0.43\textwidth}
\begin{center}
\includegraphics[height=0.9cm]{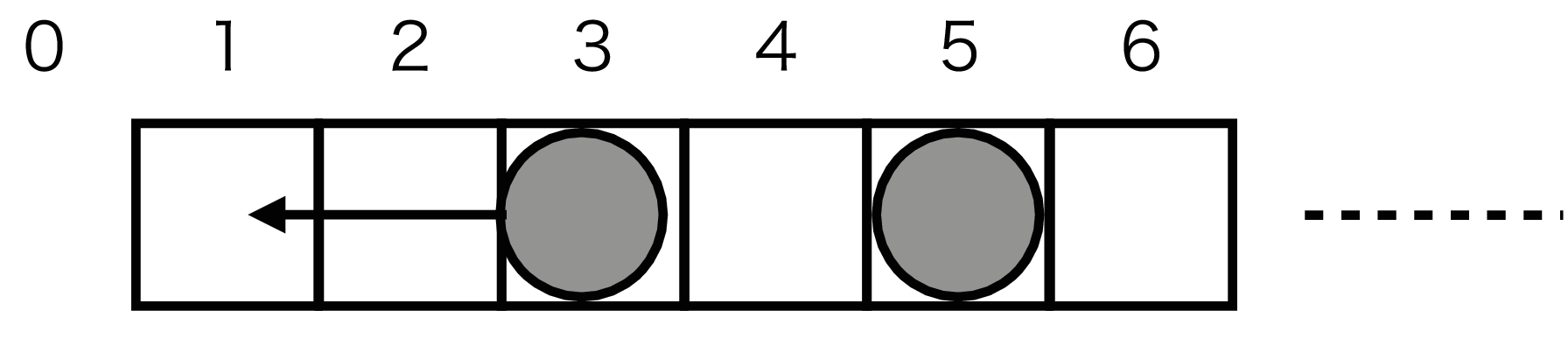}
\caption{}
\label{twocoin2}
\end{center}
\end{minipage}
\end{tabular}
\end{figure}

\begin{figure}[H]
\begin{tabular}{cc}
\begin{minipage}[t]{0.43\textwidth}
\begin{center}
\includegraphics[height=0.9cm]{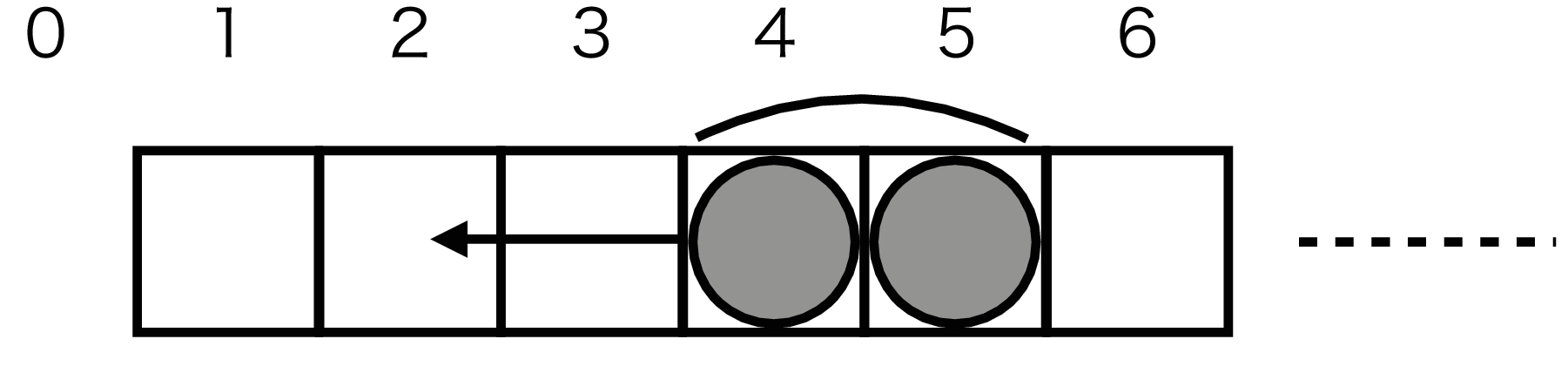}
\caption{}
\label{twocoin3}
\end{center}
\end{minipage}
\begin{minipage}[t]{0.43\textwidth}
\begin{center}
\includegraphics[height=0.9cm]{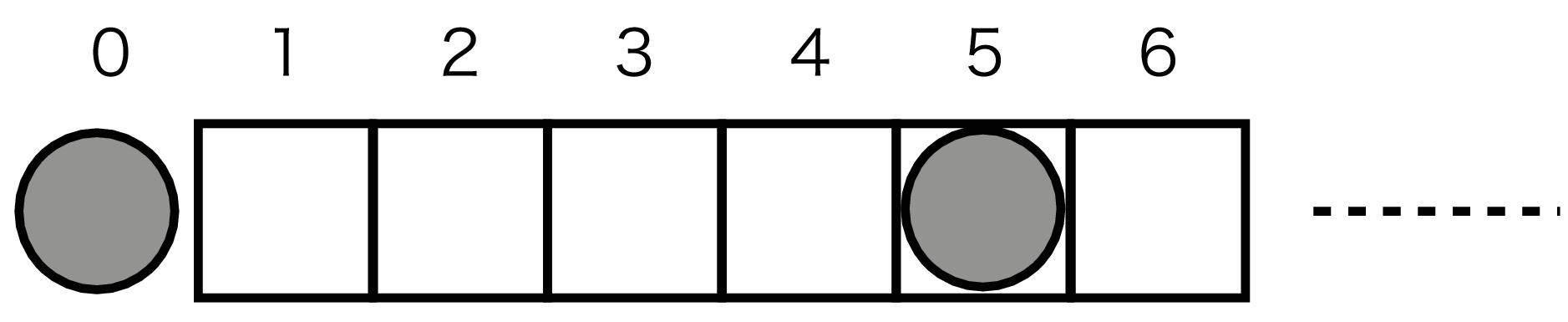}
\caption{}
\label{twocoin4}
\end{center}
\end{minipage}
\end{tabular}
\end{figure}

\begin{figure}[H]
\begin{tabular}{cc}
\begin{minipage}[t]{0.43\textwidth}
\begin{center}
\includegraphics[height=0.9cm]{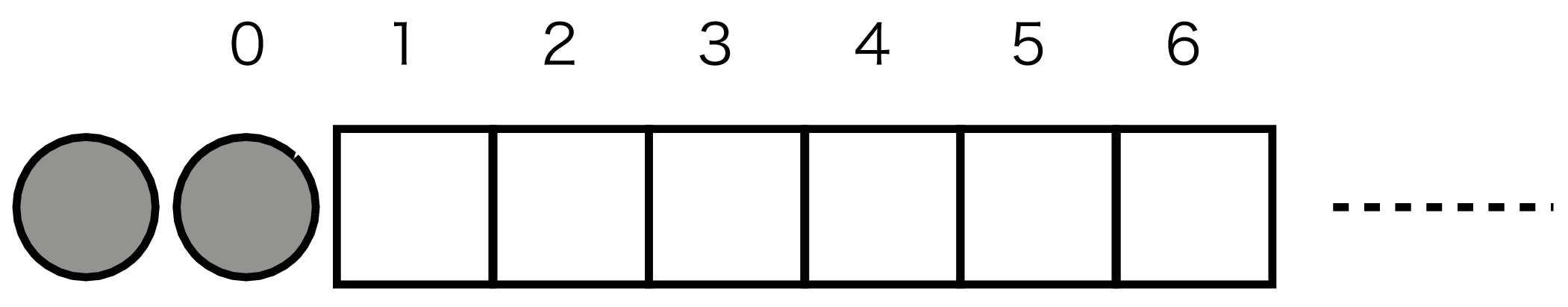}
\caption{}
\label{twocoin5a}
\end{center}
\end{minipage}
\begin{minipage}[t]{0.43\textwidth}
\begin{center}
\includegraphics[height=1.1cm]{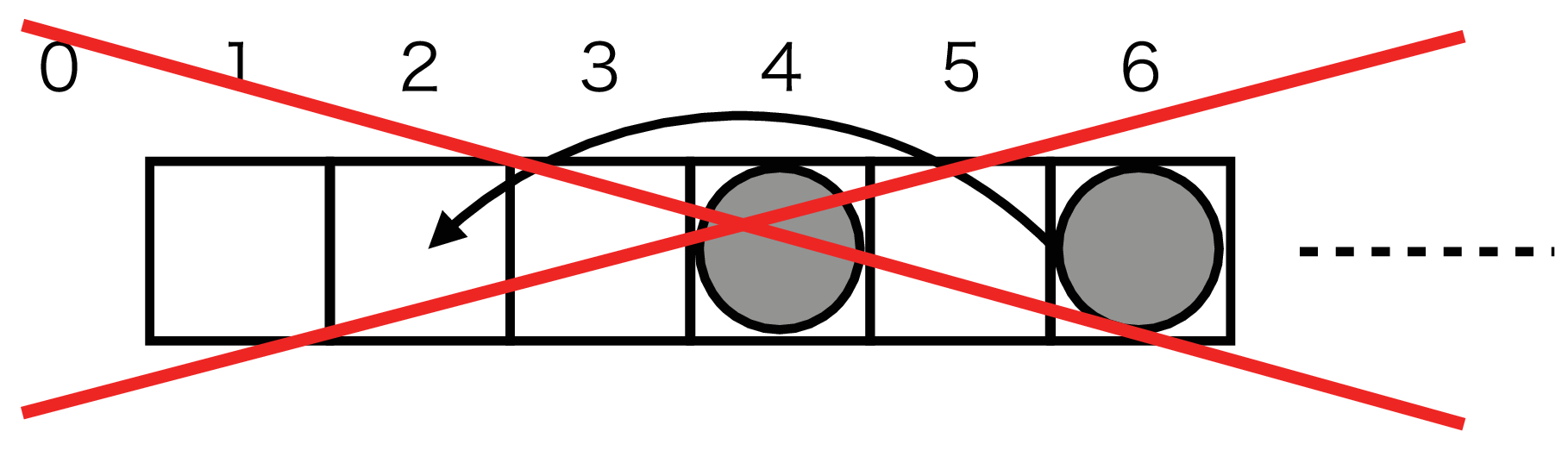}
\caption{}
\label{twocoin5}
\end{center}
\end{minipage}
\end{tabular}
\end{figure}

\begin{figure}[H]
\centering
\includegraphics[height=0.9cm]{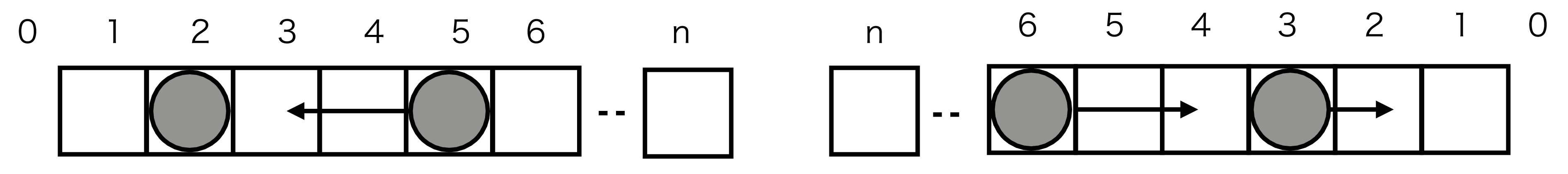}
\caption{A game with two directions}\label{twodirectiongame}
\end{figure}

\section{Combinatorial Game Theory Definitions and Theorems}\label{defandthem}
For completeness, we briefly review some of the necessary concepts of combinatorial game theory by referring to $\cite{lesson}$ and $\cite{combysiegel}$. 

\begin{definition}\label{definitionfonimsum11}
	Let $x$ and $y$ be nonnegative integers. Expressing them in base 2, 
$x = \sum_{i=0}^n x_i 2^i$ and $y = \sum_{i=0}^n y_i 2^i$ with $x_i,y_i \in \{0,1\}$.
	We define \textit{Nim-sum}, $x \oplus y$, as follows:
	\begin{equation}
		x \oplus y = \sum\limits_{i = 0}^n {{w_i}} {2^i},\nonumber
	\end{equation}
	where $w_{i}=x_{i}+y_{i} \ (\bmod\ 2)$.
\end{definition}

As chocolate bar games are impartial games without drawings, only two outcome classes are possible.
\begin{definition}\label{NPpositions}
	$(a)$ A position is referred to as a $\mathcal{P}$-\textit{position} if it is a winning position for the previous player (the player who has just moved), as long as they play correctly at each stage.\\
	$(b)$ A position is referred to as an $\mathcal{N}$-\textit{position} if it is the winning position for the next player, as long as they play correctly at each stage.
\end{definition}

\begin{definition}\label{sumofgames}
	The \textit{disjunctive sum} of the two games, denoted by $\mathbf{G}+\mathbf{H}$, is a super game in which a player may move either in $\mathbf{G}$ or $\mathbf{H}$ but not in both.
\end{definition}

\begin{definition}\label{defofmove}
	For any position $\mathbf{p}$ in game $\mathbf{G}$, a set of positions can be reached by one move in $\mathbf{G}$, which we denote as \textit{move}$(\mathbf{p})$. 
\end{definition}

\begin{definition}\label{defofmexgrundy}
	$(i)$ The \textit{minimum excluded value} ($\textit{mex}$) of a set $S$ of nonnegative integers is the least nonnegative integer that is not in S. \\
	 $(ii)$ Let $\mathbf{p}$ be the position in the impartial game. The associated \textit{Grundy number} is denoted by $G(\mathbf{p})$ and is 
 recursively defined by 
	$G(\mathbf{p}) = \textit{mex}\{G(\mathbf{h}): \mathbf{h} \in move(\mathbf{p})\}.$
\end{definition}

The next result demonstrates the usefulness of the Sprague-Brundy theory in impartial games.
\begin{theorem}\label{theoremofsumg}
Let $\mathbf{G}$ and $\mathbf{H}$ be impartial rulesets, and $G_{\mathbf{G}}$ and $G_{\mathbf{H}}$ be the Grundy numbers of game $\mathbf{g}$ played under the rules of $\mathbf{G}$ and game $\mathbf{h}$ played under those of $\mathbf{H}$. Thus, we obtain the following:\\
	$(i)$ For any position $\mathbf{g}$ in $\mathbf{G}$, 
	$G_{\mathbf{G}}(\mathbf{g})=0$ if and only if $\mathbf{g}$ is the $\mathcal{P}$-position.\\
	$(ii)$ The Grundy number of positions $\{\mathbf{g},\mathbf{h}\}$ in game $\mathbf{G}+\mathbf{H}$ is
	$G_{\mathbf{G}}(\mathbf{g})\oplus G_{\mathbf{H}}(\mathbf{h})$.
\end{theorem}
See \cite{lesson} or \cite{combysiegel} for the proof of this theorem.\\

\section{Sliding coins games}
First, we study Grundy numbers of the game of Definition \ref{definitionofstrap}, and after that we study the game of Definition \ref{definitionofstrap2}.
\begin{definition}\label{defofgrundy}
$(i)$  For  $n \in \mathbb{N}$, let
\begin{equation}
G_{n,1}=\{(a+   \frac{2(n-2)}{3}, a+  \frac{4(n-2)}{3}+2):
a \in \mathbb{Z}_{\geq0} \} \text{ when } n=2 \pmod{3}
\end{equation}
and 
\begin{equation}
G_{n,1}=\emptyset  \text{ when  } n \ne 2 \pmod{3}.
\end{equation}
$(ii)$  For  $n \in \mathbb{N}$, let
\begin{equation}
G_{n,2}=\{(a, n+ \lfloor \frac{a+1}{2} \rfloor):a \in \mathbb{Z}_{\geq0},
a+ \lfloor \frac{a+1}{2} \rfloor +1 \leq n \text{ and } \lfloor \frac{a}{2} \rfloor = n  \ (\bmod\ 2) \}.
\end{equation}
$(iii)$  For  $n \in \mathbb{N}$, let
\begin{equation}
G_{n,3}=\{(a, n- \lfloor \frac{a+1}{2} \rfloor):a \in \mathbb{Z}_{\geq0},
a+ \lfloor \frac{a+1}{2} \rfloor +1 \leq n \text{ and } \lfloor \frac{a}{2} \rfloor = n+1  \ (\bmod\ 2) \}.
\end{equation}
\end{definition}

\begin{lemma}\label{lemma1}
$(i)$ Suppose that  $n=6m+2$ for 
$m \in \mathbb{Z}_{\geq0}$. Then, we have the following.\\
\begin{equation}
G_{n,1}=\{(a+4m,a+8m+2):a \in \mathbb{Z}_{\geq0} \}.\label{lemma1i}
\end{equation}
\begin{align}
G_{n,2}= & \{(4t,6m+2t+2):t \in \mathbb{Z}_{\geq0} \text{ and } 6t \leq 6m+ 4\} \nonumber \\
& \cup \{(4t+1,6m+2t+3):t \in \mathbb{Z}_{\geq0} \text{ and } 6t \leq 6m+2 \}.
\end{align}
\begin{align}
G_{n,3}= & \{(4t+2,6m-2t+1):t \in \mathbb{Z}_{\geq0} \text{ and } 6t \leq 6m-1 \} \nonumber \\
& \cup \{(4t+3,6m-2t):t \in \mathbb{Z}_{\geq0} \text{ and } 6t \leq 6m-3 \}.
\end{align}
$(ii)$ Suppose that $n=6m+5$ for  
$m \in \mathbb{Z}_{\geq0}$. Then, we have the following.\\
\begin{equation}
G_{n,1}=\{(a+4m+2,a+8m+6):a \in \mathbb{Z}_{\geq0} \}.\label{lemma1ii}
\end{equation}
\begin{align}
G_{n,2}= & \{(4t+2,6m+2t+6):t \in \mathbb{Z}_{\geq0} \text{ and } 6t \leq 6m+1 \} \nonumber \\
& \cup \{(4t+3,6m+2t+7):t \in \mathbb{Z}_{\geq0} \text{ and } 6t \leq 6m-1 \}.
\end{align}
\begin{align}
G_{n,3}= & \{(4t,6m-2t+5):t \in \mathbb{Z}_{\geq0} \text{ and } 6t \leq 6m+4 \} \nonumber \\
& \cup \{(4t+1,6m-2t+4):t \in \mathbb{Z}_{\geq0} \text{ and } 6t \leq 6m+2 \}.
\end{align}
$(iii)$ Suppose that $n = 6m$ for
$m \in \mathbb{Z}_{\geq0}$.\\
\begin{equation}
G_{n,1}=\emptyset.
\end{equation}
\begin{align}
G_{n,2}= & \{(4t,6m+2t):t \in \mathbb{Z}_{\geq0} \text{ and } 6t \leq 6m-1 \} \nonumber \\
& \cup \{(4t+1,6m+2t+1):t \in \mathbb{Z}_{\geq0} \text{ and } 6t \leq 6m-3 \}.
\end{align}
\begin{align}
G_{n,3}= & \{(4t+2,6m-2t-1):t \in \mathbb{Z}_{\geq0} \text{ and } 6t \leq 6m-4 \} \nonumber \\
& \cup \{(4t+3,6m-2t-2):t \in \mathbb{Z}_{\geq0} \text{ and } 6t \leq 6m-6 \}.
\end{align}
$(iv)$ Suppose that $n = 6m+1$ for
$m \in \mathbb{Z}_{\geq0}$.\\
\begin{equation}
G_{n,1}=\emptyset.
\end{equation}
\begin{align}
G_{n,2}= & \{(4t+2,6m+2t+2):t \in \mathbb{Z}_{\geq0} \text{ and } 6t \leq 6m-3 \} \nonumber \\
& \cup \{(4t+3,6m+2t+3):t \in \mathbb{Z}_{\geq0} \text{ and } 6t \leq 6m-5 \}.
\end{align}
\begin{align}
G_{n,3}= &
\{(4t,6m-2t+1):t \in \mathbb{Z}_{\geq0} \text{ and } 6t \leq 6m \} \nonumber \\
& \cup \{(4t+1,6m-2t):t \in \mathbb{Z}_{\geq0} \text{ and } 6t \leq 6m-2 \}.
\end{align}
$(v)$ Suppose that $n = 6m+3$ for
$m \in \mathbb{Z}_{\geq0}$.\\
\begin{equation}
G_{n,1}=\emptyset.
\end{equation}
\begin{align}
G_{n,2}= & \{(4t+2,6m+2t+4):t \in \mathbb{Z}_{\geq0} \text{ and } 6t \leq 6m-1 \} \nonumber \\
& \cup \{(4t+3,6m+2t+5):t \in \mathbb{Z}_{\geq0} \text{ and } 6t \leq 6m-3 \}.
\end{align}
\begin{align}
G_{n,3}= &
\{(4t,6m-2t+3):t \in \mathbb{Z}_{\geq0} \text{ and } 6t \leq 6m+2 \} \nonumber \\
& \cup \{(4t+1,6m-2t+2):t \in \mathbb{Z}_{\geq0} \text{ and } 6t \leq 6m \}.
\end{align}
$(vi)$ Suppose that $n = 6m+4$ for
$m \in \mathbb{Z}_{\geq0}$.\\
\begin{equation}
G_{n,1}=\emptyset.
\end{equation}
\begin{align}
G_{n,2}= & \{(4t,6m+2t+4):t \in \mathbb{Z}_{\geq0} \text{ and } 6t \leq 6m+3 \} \nonumber \\
& \cup \{(4t+1,6m+2t+5):t \in \mathbb{Z}_{\geq0} \text{ and } 6t \leq 6m+1 \}.
\end{align}
\begin{align}
G_{n,3}= &
\{(4t+2,6m-2t+3):t \in \mathbb{Z}_{\geq0} \text{ and } 6t \leq 6m \} \nonumber \\
& \cup \{(4t+3,6m-2t+2):t \in \mathbb{Z}_{\geq0} \text{ and } 6t \leq 6m-2 \}.
\end{align}
\end{lemma}

\begin{lemma}\label{lemmaforg3set}
$(i)$ If $n=6m$, then $(4m-1,4m) \in G_{n,3}$.\\
$(ii)$ If $n=6m+1$, then $(4m,4m+1) \in G_{n,3}$.\\
$(iii)$ If $n=6m+3$, then $(4m+1,4m+2) \in G_{n,3}$.\\
$(iv)$ If $n=6m+4$, then $(4m+2,4m+3) \in G_{n,3}$.
\end{lemma}
\begin{proof}
$(i)$ Let $a =4m-1$. Then $a + \lfloor \frac{a+1}{2} \rfloor + 1 = 4m-1 + 2m+1 = 6m \leq n$, 
$\lfloor \frac{a}{2} \rfloor=\lfloor \frac{4m-1}{2} \rfloor = 2m-1$ and $n$ is odd.
Hence, $\lfloor \frac{a}{2} \rfloor = n+1  \ (\bmod\ 2)$.
\begin{equation}
(a,n-\lfloor \frac{a+1}{2} \rfloor)=(a,6m-2m)=(4m-1,4m) \in G_{n,3}.
\end{equation}
$(ii)$ Let $a =4m$. Then $a +\lfloor \frac{a+1}{2} \rfloor + 1 = 4m + 2m+1 = 6m+1 \leq n$, and 
\begin{equation}
(a,n-\lfloor \frac{a+1}{2} \rfloor)=(4m,6m-2m)=(4m,4m+1) \in G_{n,3}.
\end{equation}
$(iii)$ Let $a =4m+1$. Then $a + \lfloor \frac{a+1}{2} \rfloor + 1 = 4m+1 + 2m+1+1 = 6m+3 \leq n$, and 
\begin{equation}
(a,n-\lfloor \frac{a+1}{2} \rfloor)=(a,6m+3-(2m+1))=(4m+1,4m+2) \in G_{n,3}.
\end{equation}
$(iv)$ Let $a =4m+2$. Then $a + \lfloor \frac{a+1}{2} \rfloor + 1 = 4m+2 + 2m+1+1 = 6m+4 \leq n$, and 
\begin{equation}
(a,n-\lfloor \frac{a+1}{2} \rfloor)=(4m+2,6m+4-2m-1)=(4m+2,4m+3) \in G_{n,3}.
\end{equation}
\end{proof}

\begin{lemma}\label{fromto1}
 Suppose that  you start with a position in $G_{n,1}$ with $n=2 \pmod{3}$.\\
$(i)$ You can move to a position in $G_{n^{\prime},3}$ with $n^{\prime} < n$ and $n^{\prime} \ne 2 \pmod{3}$.\\
$(ii)$ You can move to a position in $G_{n^{\prime},1}$ with $n^{\prime} < n$ and $n^{\prime} = 2 \pmod{3}$.\\
\end{lemma}
\begin{proof}
Suppose that we start with  $(x,y) \in G_{n,1}$ with $n=2 \pmod{3}$.\\
$(i.1)$ If $n=6m+2$, by (\ref{lemma1i}) of Lemma \ref{lemma1}
$(x,y) = (a+4m,a+8m+2)$.
For $n^{\prime}=6m^{\prime}$, $n^{\prime}=6m^{\prime}+1$ with $m^{\prime} \leq m$, we can reduce  $a+8m+2$ to $4m^{\prime}$, $4m^{\prime}+1$, and by $(i)$ and $(ii)$ of Lemma \ref{lemmaforg3set}, 
move to $(4m^{\prime}-1,4m^{\prime}) \in G_{n^{\prime},3}$ and $(4m^{\prime},4m^{\prime}+1) \in G_{n^{\prime},3}$

For $n^{\prime}=6m^{\prime}+3$, $n^{\prime}=6m^{\prime}+4$ with $m^{\prime} < m$, we can reduce  $a+8m+2$ to $4m^{\prime}+2$, $4m^{\prime}+3$, and by by $(iii)$ and $(iv)$ of  Lemma \ref{lemmaforg3set}, we 
move to $(4m^{\prime}+1,4m^{\prime}+2) \in G_{n^{\prime},3}$ and $(4m^{\prime}+2,4m^{\prime}+3) \in G_{n^{\prime},3}$.\\
$(i.2)$ If $n=6m+5$, then by (\ref{lemma1ii}) of Lemma \ref{lemma1}, $(x,y)=(a+4m+2,a+8m+6)$.
For $n^{\prime}=6m^{\prime}+4$, $n^{\prime}=6m^{\prime}+3$, $n^{\prime}=6m^{\prime}+2$, $n^{\prime}=6m^{\prime}+1$ with $m^{\prime} \leq m$, we  by reducing  $a+8m+6$ to $4m^{\prime}+3$, $4m^{\prime}+2$,$4m^{\prime}+1$,$4m^{\prime}$, we move to
$(4m^{\prime}+2,4m^{\prime}+3)$, $(4m^{\prime}+1,4m^{\prime}+2)$,$(4m^{\prime},4m^{\prime}+1)$,$(4m^{\prime}-1,4m^{\prime})\in G_{n^{\prime},3}$\\
$(ii)$ Suppose that $n=3m+2$ and start with $(a+2m,a+4m+2)$. For $n^{\prime} < n$, there exists $m^{\prime}<m$ such that $n^{\prime}=3m^{\prime}+2$.
Let $a^{\prime}=2(m-m^{\prime})+a$. Then, 
$(a^{\prime}+2m^{\prime},a^{\prime}+4m^{\prime}+2) = (a+2m,a+2m+2m^{\prime}+2) \in G_{n^{\prime},1}$.
Since $a+2m+2m^{\prime}+2 < a+4m+2$, by reducing $a+4m+2$ to $a+2m+2m^{\prime}+2$, we can move to
$(a+2m,a+2m+2m^{\prime}+2) \in G_{n^{\prime},1}$.
\end{proof}

\begin{lemma}\label{fromto2}
If you start with a position in $G_{n,2}$, you
can move to a position in $G_{n^{\prime},2} \cup G_{n^{\prime},3}$.
\end{lemma}
\begin{proof}
\noindent {\tt Case 1}:
Suppose that we start with $(a,n+ \lfloor \frac{a+1}{2}  \rfloor)$ such that
$a+ \lfloor \frac{a+1}{2}  \rfloor + 1 \leq n$ and $\lfloor \frac{a}{2}  \rfloor = n  \pmod 2$.

Suppose that $a+ \lfloor \frac{a+1}{2}  \rfloor + 1 \leq n^{\prime}$.
If $n^{\prime}=n  \pmod 2$,
we can move to  $(a,n^{\prime}+ \lfloor \frac{a+1}{2}  \rfloor) \in G_{n^{\prime},2}$ with  $\lfloor \frac{a}{2}  \rfloor = n= n^{\prime} \pmod 2$.

If $n^{\prime}+1=n  \pmod 2$,
we can move to  $(a,n^{\prime}- \lfloor \frac{a+1}{2}  \rfloor) \in G_{n^{\prime},3}$ with  $\lfloor \frac{a}{2}  \rfloor = n= n^{\prime}+1 \pmod 2$.\\
\noindent {\tt Case 2}:
Suppose that 
\begin{equation}
a+ \lfloor \frac{a+1}{2}  \rfloor + 1 > n^{\prime}.\label{asmallerth}
\end{equation}
\noindent {\tt Case 2.1}:
Suppose that $n^{\prime}=0,1 \pmod{3}$. Then, there exists $m$ such that 
$n^{\prime}=6m$ or $n^{\prime}=6m+1$ or $n^{\prime}=6m+3$ or $n^{\prime}=6m+4$. 

Since 
\begin{equation}
4m-1+ \lfloor \frac{4m-1+1}{2}  \rfloor  +1 = 6m =  n^{\prime}
\end{equation}
or 
\begin{equation}
4m+ \lfloor \frac{4m+1}{2}  \rfloor  +1 = 6m+1 =  n^{\prime}
\end{equation}
or
\begin{equation}
4m+1 \lfloor \frac{4m+1+1}{2}  \rfloor  +1 = 6m+3 =  n^{\prime}
\end{equation}
or
\begin{equation}
4m+2+ \lfloor \frac{4m+2+1}{2}  \rfloor  +1 = 6m+4 = n^{\prime},
\end{equation}
by (\ref{asmallerth}), 
we have 
\begin{equation}
4m-1 < a \label{4mminus1}    
\end{equation}
or
\begin{equation}
4m < a \label{4mminus1b}    
\end{equation}
or
\begin{equation}
4m+1 < a \label{4mminus2}    
\end{equation}
or
\begin{equation}
4m+2 < a.\label{4mplus2b}   
\end{equation}

Since we have 
\begin{equation}
4m < 6m = n^{\prime} \leq n + \lfloor \frac{a+1}{2}  \rfloor 
\end{equation}
or 
\begin{equation}
4m+1 < 6m+1 = n^{\prime} \leq n + \lfloor \frac{a+1}{2}  \rfloor, 
\end{equation}
or 
\begin{equation}
4m+2 < 6m+3 = n^{\prime} \leq n + \lfloor \frac{a+1}{2}  \rfloor, 
\end{equation}
or 
\begin{equation}
4m+3 < 6m+4 = n^{\prime} \leq n + \lfloor \frac{a+1}{2}  \rfloor, 
\end{equation}
by Lemma \ref{lemmaforg3set}, (\ref{4mminus1}),(\ref{4mminus1b}),   and (\ref{4mplus2b})   
, we can move to 
$(4m-1,4m) \in G_{n^{\prime},3}$ or $(4m,4m+1) \in G_{n^{\prime},3}$ or $(4m+1,4m+2) \in G_{n^{\prime},3}$ or $(4m+2,4m+3) \in G_{n^{\prime},3}$.\\
\noindent {\tt Case 2.2}:
Suppose that $n^{\prime}=2 \pmod{3}$. Then, there exists $m$ such that 
$n^{\prime}=6m+2$ or $n^{\prime}=6m+5$. 
We prove that we can move to 
$(b+4m,b+8m+2)=(a,a+4m+2) \in G_{n^{\prime},1}$ or $(b+4m+2,b+8m+6)=(a,a+4m+4) \in G_{n^{\prime},1}$.
Since
\begin{equation}
a+\lfloor \frac{a+1}{2} \rfloor +1 \geq n^{\prime} = 6m+3 \text{ or }  6m+6,
\end{equation}
\begin{equation}
\frac{2}{3} (a+\lfloor \frac{a+1}{2} \rfloor +1) \geq  4m+2 \text{ or }  4m+4.
\end{equation}
Since
\begin{equation}
\lfloor \frac{a+1}{2} \rfloor \times 2 + 1 > \frac{2}{3} (a+\lfloor \frac{a+1}{2} \rfloor +1),
\end{equation}
\begin{equation}
n+\lfloor \frac{a+1}{2} \rfloor \geq a + \lfloor \frac{a+1}{2} \rfloor \times 2 + 1 > a+4m+2 \text{ or } a+4m+4.
\end{equation}
Therefore, 
we can move to 
$(b+4m,b+8m+2)=(a,a+4m+2) \in G_{n^{\prime},1}$ or $(b+4m+2,b+8m+6)=(a,a+4m+4) \in G_{n^{\prime},1}$.
\end{proof}

\begin{lemma}\label{fromto3}
For $n,n^{\prime}$ such that $n > n^{\prime}$,
if we start with $(x,y) \in G_{n,3}$, we can reach a position in $G_{n^{\prime},1}\cup G_{n^{\prime},2}\cup G_{n^{\prime},3}. $
\end{lemma}
\begin{proof}

Suppose that we start with $(a,n- \lfloor \frac{a+1}{2}  \rfloor)$ such that
$a+ \lfloor \frac{a+1}{2}  \rfloor + 1 \leq n$ and $\lfloor \frac{a}{2}  \rfloor = n+1  \pmod 2$.\\
\noindent {\tt Case 1}:
Suppose that 
\begin{equation}
a+ \lfloor \frac{a+1}{2}  \rfloor + 1 \leq n^{\prime}  \label{afllesstahn}  
\end{equation}
Let $a^{\prime}=a-2(n-n^{\prime})$.
Then $n-\lfloor \frac{a+1}{2}  \rfloor = n^{\prime}-\lfloor \frac{a^{\prime}+1}{2}  \rfloor $, and hence
we can move to $(a^{\prime},n-\lfloor \frac{a+1}{2}  \rfloor)$.
By (\ref{afllesstahn}), 
\begin{equation}
a^{\prime}+ \lfloor \frac{a^{\prime}+1}{2}  \rfloor + 1 \leq n^{\prime}. \label{afllesstahn2}  
\end{equation}
$\lfloor \frac{a^{\prime}}{2}  \rfloor = n^{\prime}-n+\lfloor \frac{a}{2}  \rfloor = n^{\prime} +1 \pmod 2$.
Hence $(a^{\prime},n-\lfloor \frac{a+1}{2}  \rfloor) \in G_{n^{\prime},3}$.

If $n^{\prime}=n  \pmod 2$,
we can move to  $(a,n^{\prime}+ \lfloor \frac{a+1}{2}  \rfloor) \in G_{n^{\prime},2}$ with  $\lfloor \frac{a}{2}  \rfloor = n= n^{\prime} \pmod 2$.
\noindent {\tt Case 2}:
Suppose that 
\begin{equation}
a+ \lfloor \frac{a+1}{2}  \rfloor + 1 > n^{\prime}.\label{asmallerth0}
\end{equation}
\noindent {\tt Case 2.1}:
Suppose that $n^{\prime}=0,1 \pmod{3}$. Then, there exists $m$ such that 
$n^{\prime}=6m$ or $n^{\prime}=6m+1$ or $n^{\prime}=6m+3$ or $n^{\prime}=6m+4$. 
Since
\begin{equation}
a+ \lfloor \frac{a+1}{2}  \rfloor + 1 \geq n^{\prime}+1,\label{asmallerth2}
\end{equation}
\begin{equation}
a+  \frac{a+1}{2}  + 1 \geq n^{\prime}+1.\label{asmallerth3}
\end{equation}
Hence
\begin{equation}
\frac{3}{2}a  + \frac{3}{2} \geq n^{\prime}+1.\label{asmallerth3b}
\end{equation}
For $n^{\prime}=6m,6m+1,6m+3,6m+4$,
$a+1 \geq 4m+ \frac{2}{3}$,
$a+1 \geq 4m+ \frac{4}{3}$,
$a+1 \geq 4m+ \frac{8}{3}$,
$a+1 \geq 4m+ \frac{10}{3}$, and hence
$a+1 \geq 4m+ 1$,
$a+1 \geq 4m+ 2$,
$a+1 \geq 4m+ 3$,
$a+1 \geq 4m+ 4$.
Therefore, we can move to 
$(4m-1,4m),(4m,4m+1),(4m+1,4m+2),(4m+2,4m+3)$ respectively.\\
\noindent {\tt Case 2.2}:
Suppose that $n^{\prime}=2 \pmod{3}$.    
\end{proof}

\begin{lemma}\label{fromto4}
For $n$,
$(i)$ If we start with $(x,y) \in G_{n^{\prime},1} $, we cannot reach a position in $G_{n^{\prime},2}$.

if we start with $(x,y) \in G_{n^{\prime},1} \cup G_{n^{\prime},2}\cup G_{n^{\prime},3}$, we cannot reach a position in $G_{n^{\prime},1} \cup G_{n^{\prime},2}\cup G_{n^{\prime},3}$.
\end{lemma}
\begin{proof}
$(i)$ Suppose that $n=3m+2$. We start with  $(a+2m,a+4m+2) \in G_{n,1}$ and move to $(b,n+ \lfloor \frac{b+1}{2}  \rfloor+1$ such that $b+ \lfloor \frac{b+1}{2}  \rfloor+1 \leq n$.
If $b=a+2m$, 
\begin{equation}
n+ \lfloor \frac{b+1}{2}  \rfloor+1 = \max(a+2m+ \lfloor \frac{a+2m+1}{2}  \rfloor+1,3m+2) + \lfloor \frac{2m+1}{2}  \rfloor \geq a+4m+2.
\end{equation}
\end{proof}

\begin{theorem}\label{grundygame1}
The union $G_{n,1} \cup G_{n,2}\cup G_{n,3}$ of the sets in Definition \ref{defofgrundy} the sets
is the set of positions whose Grundy number is $n$.  
\end{theorem}
\begin{proof}
This theorem is direct from Lemmas \ref{fromto1}, \ref{fromto2}, \ref{fromto3}, and \ref{fromto4}.
\end{proof}

By Theorem \ref{grundygame1}, we have formulas for the game of Definition \ref{definitionofstrap}, 
and by Theorem \ref{theoremofsumg}, we can find the $\mathcal{P}$-position of the game of Definition \ref{definitionofstrap2} by calculating the Grundy numbers.

\vspace{0.5cm}
\noindent {\bf Acknowledgement.}
 

\end{document}